\documentclass{amsart}[12pt]
\usepackage{amsmath,amsfonts,amssymb,latexsym,amsthm}
\usepackage{epsfig}
\usepackage{hyperref} 
\usepackage[dvipsnames]{xcolor}

\usepackage[normalem]{ulem}

\numberwithin{equation}{section}

\title[Breaking down the reduced Kronecker coefficients]
{Breaking down the reduced Kronecker coefficients}

\author[Igor Pak and Greta Panova]{Igor Pak$^\star$ \ and \ Greta Panova$^\dagger$}

\thanks{\ \today}
\thanks{\thinspace ${\hspace{-.45ex}}^\star$Department of Mathematics,
UCLA, Los Angeles, CA~90095.
\hskip.06cm
Email:
\hskip.06cm
\texttt{pak@math.ucla.edu}}

\thanks{\thinspace ${\hspace{-.45ex}}^\dagger$Department of Mathematics,
 USC, Los Angeles, CA~90089.
\hskip.06cm Email: \hskip.06cm
\texttt{gpanova@usc.edu}}

\newtheorem{thm}{Theorem}
\newtheorem{lemma}[thm]{Lemma}

\newtheorem{prop}[thm]{Proposition}
\newtheorem{conj}[thm]{Conjecture}

\numberwithin{equation}{section} 

\def\sq{\square}

\def\nn{\mathbb N}

\def\cc{\mathbb C}

\def\la{\lambda}
\def\ga{\gamma}
\def\si{\sigma}

\def\al{\alpha}
\def\be{\beta}

\def\cX{\mathcal X}

\def\wt{\widetilde}
\def\<{\langle}
\def\>{\rangle}

\def\GL{ {\text {\rm GL} } }

\def\0{{\mathbf 0}}

\def\SP{{\textsc{\#P}}}
\def\NP{{\textsc{NP}}}
\def\poly{{\textsc{P}}}

\def\.{\hskip.06cm}
\def\ts{\hskip.03cm}

\def\nin{\noindent}

\def\rg{\overline{g}}

\begin{document}

\begin{abstract}
We resolve three interrelated problems on \emph{reduced
Kronecker coefficients} \ts $\rg(\al,\be,\ga)$.  First, we
disprove the \emph{saturation property} which states that \ts
$\rg(N\al,N\be,N\ga)>0$ \ts implies \ts
$\rg(\al,\be,\ga)>0$ \ts for all \ts $N>1$. Second,
we esimate the maximal \ts $\rg(\al,\be,\ga)$, over all
\ts $|\al|+|\be|+|\ga| = n$.
 Finally,  we show that computing \ts $\rg(\la,\mu,\nu)$ \ts
is strongly $\SP$-hard, i.e.\ $\SP$-hard when
the input $(\la,\mu,\nu)$ is in unary.
\end{abstract}


\maketitle

\section{Introduction} \label{sec:intro}

\nin
The \emph{reduced Kronecker coefficients} were introduced by
Murnaghan in 1938 as the stable limit of \emph{Kronecker coefficients},
when a long first row is added:
$$(\circ) \quad \ \
\rg(\al,\be,\ga) \. := \. \lim_{n\to \infty} \. g\bigl(\al[n],\be[n],\ga[n]\bigr), \ \ \.
\text{where} \quad \al[n]:= (n-|\al|,\al_1,\al_2,\ldots), \ \. n\ge |\al|+\al_1\ts,
$$
see~\cite{M1,M2}.
They generalize the classical \emph{Littlewood--Richardson} (LR--) \emph{coefficients}:
$$
\rg(\al,\be,\ga) \, = \, c^\al_{\be\ga} \quad \text{for} \quad |\al|\. = \. |\be| \ts +\ts |\ga|\ts,
$$
see~\cite{Lit}.  As such, they occupy the middle ground between the Kronecker
and the LR--coefficients.  While the latter are well understood and have
a number of combinatorial interpretations, the former are notorious for
their difficulty (cf.~\cite[Prob.~2.32]{Kir}).  It is generally believed that
the reduced Kronecker coefficients are simpler and more accessible than the
(usual) Kronecker coefficients, cf.~\cite{Kir,OZ}.
The results of this paper suggest otherwise, see~$\S$\ref{ss:finrem-moral}.


\subsection{Saturation property}
The \emph{Kronecker coefficients} \ts $g(\la,\mu,\nu)$,
are defined as
$$
g(\la,\mu,\nu) \. :=\. \langle \chi^\la \ts\chi^\mu, \ts \chi^\nu\rangle \. = \. \frac{1}{n!}
\. \sum_{\si \in S_n} \chi^\la(\si)\.\chi^\mu(\si)\.\chi^\nu(\si),
$$
where \ts $\la,\mu,\nu\vdash n$, and \ts $\chi^\la$ \ts
is the irreducible character of~$S_n$ corresponding to partition~$\la$.
Similarly, the \emph{Littlewood--Richardson coefficients} are defined as
$$
c^\la_{\mu\nu} \. := \. \bigl\langle \chi^\la \.,
\. \chi^\mu\otimes\chi^\nu\uparrow^{S_n}_{S_k \times S_{n-k}}\bigr\>\,, \quad \text{where}
\quad \la\vdash n, \. \mu \vdash k, \. \nu \vdash n-k\ts.
$$
It is easy to see that \ts $c^{N \la}_{N\mu, N\nu} \ge c^\la_{\mu\nu}$ \ts
for all \ts $N\ge 1$, where \ts $N\la = (N\la_1, N\la_2,\ldots)$.
The \emph{saturation property} is the fundamental
result by Knutson and Tao~\cite{KT}, giving a converse:
$$
c^{N \la}_{N\mu, N\nu} > 0 \quad \text{for some} \. \ N\ge 1
 \quad \Longrightarrow \quad  c^\la_{\mu\nu} > 0\ts.
$$

For a partition $\al\vdash k$ and $n\ge k+\al_1$, we have \ts
$\al[n]= (n-k,\al_1,\al_2,\ldots)\vdash n$. It is known that \ts
$g(\al[n+1],\be[n+1],\ga[n+1])\ts \ge \ts g(\al[n],\be[n],\ga[n])$
\ts for all~$n$, whenever the right hand side is defined.  In this notation,
Murnaghan's result~$(\circ)$ states that
\ts $\rg(\al,\be,\ga) = g(\al[n],\be[n],\ga[n])$ \ts
for $n$ large enough.

The saturation property fails for the Kronecker coefficients, i.e.
$g(2^2,2^2,2^2)=1$ \ts but \ts $g(1^2,1^2,1^2)=0$. It is a long-standing
open problem whether it holds for the reduced Kronecker coefficients.
This was independently conjectured in 2004 by Kirillov \cite[Conj.~2.33]{Kir}
and Klyachko \cite[Conj.~6.2.4]{Kly}~:

\begin{conj}[{Kirillov, Klyachko}]\label{conj:sat}
The reduced Kronecker coefficients satisfy the saturation property:
$$\rg(N\al,N\be, N\ga) > 0 \quad \text{for some} \. \ N\ge 1
\quad \Longrightarrow \quad  \rg(\al,\be,\ga) > 0\ts.
$$
\end{conj}

This conjecture was motivated by the known converse:
$$\rg(\al,\be,\ga)\. > \.0   \quad \Longrightarrow \quad  \rg(N\al,N\be, N\ga) \. > \. 0
\ \ \ \text{for all} \ \ N\ge 1\ts,
$$
see below. Here is the first result of this paper.

\begin{thm}\label{t:sat}
For all \ts $k\ge 3$, the triple of partitions
\ts $\bigl(1^{k^2-1},1^{k^2-1},k^{k-1}\bigr)$ \ts is a counterexample
to Conjecture~\ref{conj:sat}.  Moreover, for every partition~$\ga$ \ts s.t.\ \ts
$\ga_2\ge 3$,
there are infinitely many pairs \ts $(a,b)\in \nn^2$ \ts for which the triple of partitions
\ts $(a^b,a^b,\ga)$ \ts is a counterexample to Conjecture~\ref{conj:sat}.
\end{thm}

These results both contrast and complement~\cite[Cor.~6]{CR}, which
confirms the saturation property for triples of the form \ts $(a^b,a^b,a)$.


\subsection{Maximal values}
Our second result is a variation on Stanley's recent bound on the maximal
Kronecker and LR--coefficients:

\begin{thm}[{\cite{Sta-kron,EC2-supp}, see also~\cite{PPY}}] \label{t:stanley}
We have:
$$\aligned
(\ast) \qquad \quad \ts\  & \ \, \max_{\la\vdash n} \, \max_{\mu\vdash n} \,  \max_{\nu\vdash n}
\,\,\, g(\la,\mu,\nu) \  = \, \sqrt{n!} \,\. e^{-O(\sqrt{n})}\ts, \\ 
(\ast\ast)  \qquad \quad \, & \ts \max_{0\le k\le n} \. \max_{\la\vdash n} \. \max_{\mu\vdash k} \.
\max_{\nu\vdash n-k}  \,\. c^\la_{\mu,\nu} \,  = \, 2^{n/2 \ts \ts - \ts O(\sqrt{n})}\ts.
\endaligned
$$
\end{thm}

In~\cite{PPY}, we prove that the maximal Kronecker and LR--coefficients
appear when all three partitions have near-maximal dimension,
which in turn implies that they have a \emph{Vershik--Kerov--Logan--Shepp}
(VKSL) \emph{shape}. See also~\cite{PP2} for refined upper bounds on
(reduced) Kronecker coefficients with few rows. Here we obtain
the following analogue of Stanley's Theorem~\ref{t:stanley}.

\begin{thm} \label{t:kron-max}
We have:
$$\max_{a+b+c\le 3n}
\max_{\al\vdash a} \.\, \max_{\be \vdash b} \. \max_{\ga \vdash c}
\,\,\. \rg(\al,\be,\ga) \,  = \,\. \sqrt{n!} \,\. e^{O(n)}\ts.
$$
\end{thm}

\smallskip

\subsection{Complexity}
Our final result is on complexity of computing the reduced
Kronecker coefficients.  Via reduction to LR--coefficients,
computing the reduced Kronecker coefficients is classically
$\SP$-hard, see~\cite{Nar}.  The following recent result
by Ikenmeyer, Mulmuley and Walter is a far-reaching extension:

\begin{thm}[{\cite{IMW}, cf.~$\S$\ref{ss:finrem-IMW}}] \label{t:IMW}
Computing the Kronecker coefficients
\ts $g(\la,\mu,\nu)$ \ts is \ts \emph{strongly} \ts $\SP$-hard.
\end{thm}

Here by \emph{strongly $\SP$-hard} we mean $\SP$-hard when the
input $(\la,\mu,\nu)$ is given in unary.  In other words,
the input size of the problem is the total number
of squares in the three Young diagrams.  The theorem is in sharp
contrast with computing \ts
$\chi^{(n-k,k)}[\la]$ \ts which is
$\SP$-complete but not strongly $\SP$-complete, see~\cite[$\S$7]{PP}.

\begin{thm}\label{t:SP-red}
Computing the reduced Kronecker coefficients \ts $\rg(\al,\be,\ga)$
\ts is \ts \emph{strongly} \ts $\SP$-hard.
\end{thm}

Let us mention that the problem of computing the
(reduced) Kronecker coefficients is not known to be in~$\SP$, see~\cite{PP}.
In fact, finding a combinatorial interpretation for (reduced)
Kronecker coefficients is a classical open problem~\cite[Prob.~10]{Sta-OP}.
Note also that Theorem~\ref{t:SP-red} is stronger than Theorem~\ref{t:IMW},
since in the limit~$(\circ)$ it suffices to take \ts
$n  \ts \ge \ts |\al| \ts + \ts |\be| \ts + \ts |\ga|$,
see~\cite{BOR2,Val}. Indeed, this shows that the
reduced Kronecker coefficient problem is a subset
of instances of the usual Kronecker coefficient problem 
(cf., however,~$\S$\ref{ss:finrem-Iken}).

\medskip

\section{Disproof of the saturation property}

\subsection{Preliminaries}
We assume the reader is familiar with basic results and standard
notations in Algebraic Combinatorics, see~\cite{Sag,EC2}.
We also need the following two results on Kronecker coefficients.

\begin{lemma}[Symmetries] \label{l:sym}
For every \ts $\la,\mu,\nu \vdash n$, we have:
$$
g(\la,\mu,\nu)\. = \. g(\la',\mu',\nu)\. = \. g(\mu,\la,\nu)\. = \. g(\la,\nu,\mu)\ts.
$$
\end{lemma}

\begin{lemma}[Semigroup property~\cite{CHM,Manivel}]\label{l:semi}
Suppose $\al,\be,\ga \vdash m$, such that \ts $g(\al,\be,\ga)> 0$.
Then, for all partitions \ts $\la,\mu,\nu \vdash n$, we have:
$$g(\la+\al,\mu+\be,\nu+\ga)\. \geq \. g(\la,\mu,\nu)\ts.
$$
\end{lemma}

This result is crucial for understanding of reduced Kronecker coefficients.  First,
since \ts $g(1,1,1)=1$, we conclude that the sequence
\ts $\bigl\{g(\al[n],\be[n],\ga[n])\bigr\}$ \ts is weakly increasing with~$n$.
Similarly,  the sequence \ts $\bigl\{g(N\la,N\mu,N\nu)\bigr\}$  \ts is
weakly increasing with~$N$ if $g(\la,\mu,\nu)>0$.

\smallskip

Let \ts $\ell(\la)$ be the \emph{number of parts} of the partition~$\la$, and \ts
$d(\la) := \max\{k: \la_k \geq k\}$ \ts be the \emph{Durfee size}.

\begin{lemma}[\cite{Dvir}]\label{l:dvir}
Let $\la,\mu,\nu \vdash n$ be such that $g(\la,\mu,\nu) >0$. Then \ts
$d(\la) \ts \le \ts 2\ts d(\mu) \ts d(\nu)$.
\end{lemma}

\nin
The following argument gives a blueprint for the proof of Theorem~\ref{t:sat}.

\begin{prop}  \label{p:ex}
Let $\al=1^5$, $\ga=3^2$.  Then \ts
$\rg(\al,\al,\ga) = 0$, but \. $\rg(2\al,2\al,2\ga)>0$.
\end{prop}

\begin{proof}  
First, let us show that \ts $g(\al[n],\al[n],\ga[n])=0$ \ts for all $n \geq 9$.
Indeed, we have \ts $d(\al[n])=1$, and \ts $d(\ga[n]) =3 > 2 \ts d(\al[n])^2=2$,
and the claim follows from Lemma~\ref{l:dvir}. \ts
On the other hand, a direct calculation shows that \ts
$g(2\al\ts [18],2\al\ts [18]), 2\ga\ts [18]) = g(8\ts 2^5, 8\ts 2^5, 6^3) =8$,
which implies \ts $\rg(2\al,2\al,2\ga) \geq 8$.\footnote{In fact, a longer direct calculation gives \ts $\rg(2\al,2\al,2\ga)=12$.  }
This contradicts the saturation property in
Conjecture~\ref{conj:sat} for \ts $N=2$.
\end{proof}

\subsection{Proof of Theorem~\ref{t:sat}}
We prove the first statement of the theorem. Let $k\ge 3$, and let \ts $\al =\bigl(1^{k^2-1}\bigr)$,
\ts $\ga = \bigl(k^{k-1}\bigr)$ \ts be as in the theorem.
%
Since \ts $d(\al[n])=1$ \ts and \ts
$d(\ga[n]) = k$ \ts for all~$n\ge k^2$, we have \ts
$2\ts d(\al[n])^2 = 2 < d(\ga[n])=k$. Thus, we have
\ts $\rg(\al,\al,\ga)=0$ \ts
by Lemma~\ref{l:dvir}.

\smallskip

\begin{lemma}[\cite{BB}]
Let \ts $\la=\la'$ \ts be a self-conjugate partition.  Then \ts
$g(\la,\la,\la)>0$. \label{l:BB}
\end{lemma}

\smallskip

\nin
By Lemma~\ref{l:BB}, the symmetry and semigroup properties
(Lemma~\ref{l:sym} and~\ref{l:semi}),
we have:
$$\aligned
\rg(k \al, k\al, k\ga) \. & = \.  \rg\bigl(k^{k^2-1}, k^{k^2-1}, (k^2)^{k-1}\bigr) \. \geq \.
g\bigl(k^{k^2-1}[k^3], k^{k^2-1}[k^3], (k^2)^{k-1}[k^3]\bigr) \\
& \geq \. g\bigl(k^{k^2},k^{k^2},(k^2)^k\bigr)  \. = \.
g\bigl((k^2)^k, (k^2)^k, (k^2)^k\bigr) \. \geq \. g\bigl(k^k,k^k,k^k\bigr) \. > \. 0\ts.
\endaligned
$$
This contradicts the saturation property in Conjecture~\ref{conj:sat} for $N=k$,
and proves the first part of the theorem. \ts
For the second part, we need the following more technical result:

\begin{lemma}[{\cite[Thm~1.10]{IP}}]\label{l:IP}
Let \ts $\cX:= \{1, \. 1^2, \. 1^4, \. 1^6, \. 2\ts 1, \. 3\ts 1\}$,
and let partition \ts $\nu \notin \cX$.  Denote \ts $\ell := \max\{\ell(\nu)+1,9\}$,
and suppose \ts $r >  3\ell^{3/2}$, \ts $s \geq 3 \ell^2$,
and \ts $|\nu| \leq rs/6$. Then \ts $g(s^r, s^r,\nu[rs])>0$.
\end{lemma}

We construct the counterexample based on Lemma~\ref{l:IP}. For a partition~$\ga$,
let \ts $\ell:=\max\{\ell(\ga)+1,\. 9\}$ as in the lemma.  Let
\ts $b \ts \geq \ts \max\bigl\{ 3\ell^{3/2}, \, |\ga|\ts /\ts (6\sqrt{d(\ga[n])/2}-6) \bigr\}$.
Since \ts $\ga_2\ge 3$, we have \ts $d(\ga[n])\ge 3$. Thus, there exists at
least one \ts $a\ge 1$, such that \ts
$|\ga|/(6b) \ts \le  a < \sqrt{d(\ga[n])/2}$.  Let us show now that \ts $(a,b)$ \ts is a pair
as in the theorem.

Take \ts $\al:=(a^b)$. Since \ts $d(\al[n]) \leq a$,
we have \ts
$2\ts d(\al[n])^2 \ts \leq \ts 2\ts a^2 \ts < \ts d(\ga[n])$.
Thus, we have \ts $\rg(\al,\al,\ga)=0$ \ts by Lemma~\ref{l:dvir}.

On the other hand, let $N \geq 3\ell^2/a$,\, $\nu:=N\ga$, \ts $r := b+1$, and \ts $s := Na$.
Then \ts $|\nu| \ts \leq \ts Nab/6  \ts < \ts  rs/6$, \ts $r>3\ell^{3/2}$, and \ts
$s =Na \geq 3\ell^2$, by construction. Since $\nu \not\in \cX$  for all $N>1$,
the conditions of Lemma~\ref{l:IP} are satisfied.  We conclude:
$$
\rg(N\al, N\al, N\ga) \. = \.
\rg\bigl(Na^b, Na^b, N\ga\bigr) \. \geq \. g\bigl(s^{b+1}, s^{b+1}, N\ga\ts [rs]\bigr)
\. = \. g\bigl(s^r,s^r,\nu[rs]\bigr) \.  > \.0 \ts,
$$
which implies that \ts $(\al,\al,\ga)$ \ts is a counterexample to the
saturation property.  Since the construction works for all \ts $b$ \ts large
enough as above, this proves the second part of the theorem. \ $\sq$

%
%

\medskip

\section{Bounds and complexity via identities}

\subsection{Proof of Theorem~\ref{t:kron-max}}

We follow~\cite{PPY} in our exposition.  We start with the following identity
\cite[Cor.~4.5]{BDO}:
\begin{equation}\label{eq:BDO}
\rg(\al,\be,\ga) \, = \,
\sum_{m=0}^{\lfloor k/2\rfloor} \, \sum_{\pi\vdash q+m-b} \, \sum_{\rho\vdash q+m-a} \, \sum_{\si \vdash m}\,
\sum_{\la,\mu,\nu\vdash k-2m} \ c^{\al}_{\nu\pi\rho} \, c^{\be}_{\mu\pi\si} \, c^{\ga}_{\la\rho\si} \, g(\la,\mu,\nu)
\ts,
\end{equation}
where $a=|\al|$, $b=|\be|$, $q=|\ga|$, $k=a+b-q$, and
$$
c^\la_{\al\ts \be\ts\ga} \. = \, \sum_{\tau} \. c^{\la}_{\al\tau} \ts c^{\tau}_{\be\ga}\ts.
$$
For the upper bound, by~\cite[Thm~1.5]{PPY} which extends~$(\ast\ast)$ in
Theorem~\ref{t:stanley}, we have:
$$c^\la_{\al\be} \. \le \. \binom{N}{a}^{1/2} \quad \ \text{for all} \ \ \la\vdash N, \,\. \al\vdash a, \,\. \be \vdash N-a\ts.
$$
Using the \emph{Vandermonde identity} for the sums of binomial coefficients, we have:
$$c^\la_{\al\ts \be\ts\ga} \. \le \. \binom{N}{a,b,N-a-b}^{1/2} \, \le \. 3^{N/2} \quad \text{for all}
\ \ \la\vdash N, \,\.\al\vdash a, \,\. \be \vdash b, \,\.\ga\vdash N-a-b\ts.
$$
In this notation, the theorem is a maximum over \ts $a+b+q \le 3n$.
Combining these with~$(\ast)$ in Theorem~\ref{t:stanley},
we have:
$$
\rg(\al,\be,\ga) \, \le \, (3\ts n/2) \cdot p(3n)^6 \cdot 3^{3n/2} \cdot \sqrt{n!}
\, = \, \sqrt{n!} \,\. e^{O(n)}\ts.
$$
\nin
For the lower bound, let \ts $\al,\be,\ga\vdash n$, and
note that \ts $\rg(\al,\be,\ga) \ge g(\al,\be,\ga)$, which is achieved
in~\eqref{eq:BDO} for $m=0$.  The result
now follows from part~$(\ast)$ of Theorem~\ref{t:stanley}.
\ $\sq$

\smallskip

\subsection{Proof of Theorem~\ref{t:SP-red}}
Let $\la=(\la_1,\la_2,\ldots)\vdash n$, which can be viewed
as an infinite nonincreasing sequence by appending zeros at the end.
Denote \ts $\wt\la := (\la_2,\la_3,\ldots)$.  For all $i\ge 1$, define
$$\la^{\<i\>} \. := \. \bigl(\la_1+1,\la_2+1,\ldots, \la_{i-1}+1,\la_{i+1},\la_{i+2},\ldots\bigr),
$$
so in particular \ts $\la^{\<1\>}=\wt\la$.
The result is a direct consequence of the following identity:
\begin{equation}\label{eq:BOR}
g(\la,\mu,\nu) \. = \, \sum_{i=1}^{\ell(\mu)\ts\ell(\nu)} \. (-1)^i \,\, \rg\bigl(\la^{\<i\>},\wt\mu,\wt\nu\bigr)\ts,
\end{equation}
see~\cite[Thm~1.1]{BOR2}.
From Theorem~\ref{t:IMW}, computing $g(\la,\mu,\nu)$ is \ts $\SP$-hard in unary.
The identity~\eqref{eq:BOR} has polynomially many terms, and thus gives a polynomial
reduction.  \ $\sq$

\medskip

\section{Final remarks and open problems}\label{s:finrem}

\subsection{} \label{ss:finrem-moral}
All three results in this paper are centered
around the same (philosophical) claim, that the reduced Kronecker
coefficients are closer in nature to the (usual) Kronecker
coefficients than to the LR--coefficients.  This is manifestly
evident from both the statements and the proofs of the theorems.
In fact, the only clue we know of the difference is the result
in~\cite{Ent}. \ts 
However, this claim should
not be taken as a suggestion that the LR--coefficients are not
strongly $\SP$-hard.  We do, in fact, conjecture that
computing $c^\la_{\mu\nu}$ is strongly $\SP$-hard \cite[Conj.~8.1]{PP},
but this remains beyond the reach of existing technology.

\subsection{} \label{ss:finrem-gen}
There is a general setting which extends the
stability of Kronecker coefficients to other families
of stable limits, see~\cite{SS}.  Manivel asks if
the saturation property holds for all these families,
but notes that ``we actually have only very limited
evidence for that'' \cite{Manivel}.  In view of our
results, it would be interesting to see if the
saturation property holds for \emph{any} of these
families of stable coefficients.

\subsection{}\label{ss:finrem-IMW}
Theorem~\ref{t:IMW} is not stated in~\cite{IMW} in this
form.  It does however follow directly from the proof,
which is essentially a parsimonious reduction from
the {\sc 3-Partition} problem classically known to
be (strongly) $\NP$-complete, and thus the counting is
(strongly) $\SP$-complete.

\subsection{}
By analogy with~\cite{PPY}, it would be interesting to
find the asymptotic limit shape of partitions \ts
$\al,\be,\ga$ \ts which achieve a maximum in
Theorem~\ref{t:kron-max}. We believe the bounds in
our proof of the theorem can be improved to show that
all three shapes are Plancherel of the same size.

\subsection{} \label{ss:finrem-Mulmuley}
Among other consequences, the saturation property
implies that the \emph{vanishing problem} \ts
$c^\la_{\mu\nu}>^?0$ \ts is in~$\poly$, see~\cite{MNS}.
The main result of~\cite{IMW} proved that the vanishing
problem \ts $g(\la,\mu,\nu)>^?0$ \ts is $\NP$-hard,
refuting Mulmuley's conjecture (see e.g.~\cite[$\S$2]{PP}).
Following the pattern in~$\S$\ref{ss:finrem-moral} above,
we conjecture that the vanishing problem \ts $\rg(\al,\be,\ga)>^?0$
\ts for reduced Kronecker coefficients is also $\NP$-hard.

\subsection{} \label{ss:finrem-Iken}
There is a subtle but important technical differences between 
the way we state Theorem~\ref{t:IMW} and the way it is stated 
in~\cite{IMW}.  While we implicitly use the (standard) 
\emph{Turing reduction} to derive Theorem~\ref{t:SP-red} 
from Theorem~\ref{t:IMW}, 
the original proof in~\cite{IMW} 
uses a more restrictive \emph{many-to-one reduction} 
(sometimes called \emph{Karp reduction}). 
Such a reduction for Theorem~\ref{t:SP-red} would also resolve
our conjecture above on the vanishing problem for the reduced 
Kronecker coefficients. 

\subsection{} In an appendix to~\cite{BOR1}, Mulmuley
conjectures a weaker property than the saturation,
for the stretched Kronecker
coefficients, which would in turn imply a polynomial result for
a closely related complexity problem.  See also~\cite{Kir} for further
saturation related conjectures.


\vskip.4cm
{\small

\noindent
{\bf Acknowledgements.} \
The authors are grateful to Christine Bessenrodt, Chris Bowman and
Rosa Orellana for interesting conversations
and helpful remarks.  Special thanks to Christian Ikenmeyer
for the explanation of the inner working of~\cite{IMW}, and to
Mike Zabrocki for help with the computer algebra. The results
of the paper were obtained during authors' back to back visits
at the Oberwolfach Research Institute for Mathematics and the
Mittag-Leffler Institute; we are grateful for their hospitality.
Both authors were partially supported by the NSF.}

\vskip.7cm


\end{document}